\documentclass[letterpaper, 10 pt, conference]{ieeeconf}  

\IEEEoverridecommandlockouts                              
\usepackage{amsmath,graphicx}
\usepackage{amssymb}
\usepackage{amsfonts}
\usepackage{comment}
\let\labelindent\relax
\usepackage{enumitem}
\usepackage{algorithm ,algpseudocode}
\let\oldalgorithmic\algorithmic 
\renewcommand\algorithmic{\ttfamily\fontseries{l}\selectfont\oldalgorithmic}
\usepackage{lmodern}
\usepackage{multirow} 

\usepackage[english]{babel}
\usepackage{subcaption}
\usepackage[long]{optidef}
\usepackage{cite}

\newtheorem{theorem}{\textbf{Theorem}}
\newtheorem{lemma}{\textbf{Lemma}}
\newtheorem{corollary}{\textbf{Corollary}}
\newtheorem{remark}{\textbf{Remark}}
\newtheorem{assumption}{\textbf{Assumption}}
\renewenvironment{proof}[1][\textbf{Proof}]{{\noindent\it #1:}\enspace}{\hfill $\blacksquare$\par}
\newenvironment{rproof}[1][\textbf{Proof}]{{\noindent\it #1:}\enspace}{\hfill \par}

\newcommand{\mycaptionof}[2]{\captionof{#1}{#2}}

\newcommand{\NOTE}[1]{\Ensure \hspace{-0.15in} \textbf{\textit{#1}}}

\usepackage[colorlinks,citecolor=red,urlcolor=blue,bookmarks=false,hypertexnames=true]{hyperref}
\def\x{{\mathbf x}}

\def\z{{\mathbf z}}
\def\y{{\mathbf y}}
\def\blambda{{\boldsymbol \lambda}}
\def\bphi{{\boldsymbol \phi}}
\def\btheta{{\boldsymbol \theta}}
\def\balpha{{\boldsymbol \alpha}}
\def\bbeta{{\boldsymbol \beta}}
\def\s{{\mathbf s}}
\def\g{{\mathbf g}}
\def\R{{\mathbb R}}

\def\E{{\mathbb E}}
\def\e{{\mathbf e}}
\def\v{{\mathbf v}}

\DeclareMathOperator*{\argmin}{argmin}
\allowdisplaybreaks[4]

\title{\LARGE \bf
Communication-efficient distributed optimization with adaptability to system heterogeneity
}

\author{Ziyi Yu and Nikolaos M. Freris
\thanks{The authors are with the School of Computer Science, University of Science and Technology of China, Hefei, China \newline
        {\tt\small Emails:  yuziyi@mail.ustc.edu.cn, nfr@ustc.edu.cn}\newline
        This research was supported by the USTC Research Department under grant WK2150110025. Correspondence to N. Freris.
        }%
}

\begin{document}

\maketitle
\thispagestyle{empty}
\pagestyle{empty}

\begin{abstract}
We consider the setting of agents cooperatively minimizing the sum of local objectives plus a regularizer on a graph. 
This paper proposes a primal-dual method in consideration of three distinctive attributes of real-life multi-agent systems, namely: (i)~expensive communication, (ii)~lack of synchronization, and (iii)~system heterogeneity.  
In specific, we propose a distributed asynchronous algorithm with minimal communication cost, in which users commit variable amounts of local work on their respective sub-problems. 
We illustrate this both theoretically and experimentally in the machine learning setting, where the agents hold private data and use a stochastic Newton method as the local solver. 
Under standard assumptions on Lipschitz continuous gradients and strong convexity, our analysis establishes linear convergence in expectation and characterizes the dependency of the rate on the number of local iterations. 
We proceed a step further to propose a simple means for tuning agents’ hyperparameters locally, so as to adjust to heterogeneity and accelerate the overall convergence.
Last, we validate our proposed method on a benchmark machine learning dataset to illustrate the merits in terms of computation, communication, and run-time saving as well as adaptability to heterogeneity.
\end{abstract}

\section{INTRODUCTION}
\label{sec:intro}

The resurgence of distributed optimization \cite{YANG2019278} is underscored by the vast numbers of smart devices (phones, sensors, etc.) that collect massive volumes of data in a distributed fashion.
It has concise advantages compared to its centralized counterpart in terms of balancing computation and communication costs as well as minimizing delays, and has thus found applications in a broad range of fields, such as in machine learning and data mining \cite{li2022HIPPO, vlachos2015compressive}, signal processing \cite{sopasakis2016accelerated, FedZO}, wireless sensor networks \cite{freris2010fundamentals}, and control \cite{Conte2012Computational}.

The archetypal problem that we consider is to:
\begin{mini}
    {\hat \x\in \R^d}{\frac{1}{n}\sum_{i=1}^n f_i(\hat \x) + r(\hat \x),}
    {\label{eqn:DistObjDef}}{}
\end{mini}
where $n$ is the number of agents in the network, $f_i(\cdot): \R^d \rightarrow \R$ is a strongly convex and smooth cost function pertaining to agent $i \in [n]$ ($[n]:=\{1,\hdots,n\})$, and $r(\cdot)$ is a convex but possibly non-smooth function that serves to impose regularity on the solution.
We further pose additional structure on the cost function as motivated by the machine learning framework.
In specific, the $i$-th agent possesses a batch of data samples $\big\{\s_i^j\big\}_{j=1}^{D_i}$ with volume $D_i$, whence:
\begin{equation}\label{eqn:localloss}
    f_i(\hat \x) := \frac{1}{D_i} \sum_{j=1}^{D_i} l(\s_i^j;\hat \x),
\end{equation}
where $l(\cdot;\hat \x)$ is the loss function associated with one data sample for a model parameter $\hat \x$.

There has been extensive work on the \emph{consensus} framework, especially first-order methods \cite{nedic2009distributed, shi2015extra, shi2015proximal, qu2017harnessing, nedic2017achieving, alghunaim2019P2D2}, which are prone to slow convergence in ill-conditioned problems.
Second-order information can be used to remedy this \cite{wei2013distributed, mokhtari2016network, crane2019dingo}, nonetheless, this comes with caveats in terms of extensive message-passing and stringent synchronization requirements.

Our proposed solution falls under the taxonomy of primal-dual methods, in particular, the celebrated \emph{Alternating Direction Method of Multipliers} (ADMM) \cite{boyd2011distributed} which has given rise to a multitude of distributed methods \cite{wei2012DADMM, shi2014linear, ling2015dlm, chang2014multi, Chang2016asyn} that feature fast convergence and robustness to hyperparameter selection.
The computationally challenging component of ADMM lies in solving optimization sub-problems associated with the local objectives. 
To this end, there has been extensive work on inexact methods, that typically involves a gradient step for the local problems \cite{ling2015dlm, chang2014multi}. 
In contrast, the use of second-order information for obtaining inexact solutions has been unattractive as it requires inner communication loops to approximate the Hessian \cite{Mokhtari2016DADMM, Eisen2019PD}. 
This issue was resolved by the DRUID framework \cite{li2022DRUID}, which re-formulated the problem by introducing intermediate variables, thus allowing for a decomposition that supports (quasi-)Newton local solvers without an increase in communication costs. 
Nevertheless, the methods in \cite{li2022DRUID} only considered a single local iteration and deterministic local solvers, while the analysis heavily relies on this setting; they are thus unsuitable to capture computational heterogeneity in a real system. 

The proposed solution in this paper is motivated by three main features/requirements in large-scale smart systems, namely:
\begin{enumerate}
    \item \emph{communication is costly} (in view of battery drainage, limited bandwidths, and stringent delay considerations),
    \item \emph{synchronous methods are unattractive} (due to agent unavailability caused by varying operating conditions as well as the difficulty of the synchronization problem \cite{Freris2011Asy}), and 
    \item \emph{heterogeneity in computing capabilities} (this allows for variable work loads to be committed based on the device hardware and battery levels). 
\end{enumerate}
We propose a distributed asynchronous communication-efficient protocol tailored for variability to system heterogeneity. 
The contributions are summarized as follows.

\noindent \textbf{Contributions:} 
\begin{enumerate}
    \item The proposed method, DRUID-VL, allows an arbitrary subset of agents to participate in each round and features minimal communication costs (a single broadcasting step of the local variable to neighboring agents).
    \item It specifically addresses heterogeneity in terms of local work. We illustrate this in the machine learning setting by applying an agent-dependent number of stochastic Newton steps for the local sub-problem. 
    \item  We establish linear convergence in expectation (for arbitrary levels of heterogeneity) and characterize the convergence rate (\textit{Theorem} \ref{thm:main}). In particular, our analysis captures the effect of local work on the convergence of the algorithm and allows for a simple method to tune hyperparameters locally to further accelerate the convergence.
    \item We experimentally demonstrate on a real-life dataset the communication and computation benefits of our method vis-à-vis prior art, and also study the algorithm's gains by adapting to heterogeneity.
\end{enumerate}

\section{PRELIMINARIES}
\label{sec:Prel}

\subsection{Problem Formulation}
\label{subsec:reform}
The network topology is captured by an undirected graph $\mathcal{G} = (\mathcal{V}, \mathcal{E})$, where $\mathcal{V} = \left\{1, \hdots, n\right\}$ is the vertex set and $\mathcal{E} \subseteq \mathcal{V} \times \mathcal{V}$ is the edge set, i.e. $(i, j) \in \mathcal{E}$ if agent $i$ can exchange packets with agent $j$ ($\mathcal{N}_i:=\{j: (i,j)\in\mathcal{E}\}$ denotes the neighborhood of $i$).
Provided that $\mathcal{G}$ is connected, \eqref{eqn:DistObjDef} can be directly recast in the following consensus setting:
\begin{mini}
    {\x_i, \btheta, \z_{ij}}{\frac{1}{n}\sum_{i=1}^n f_i(\x_i) + r(\btheta),}
    {\label{eqn:DistObjConsensusEqua}}{}
    \addConstraint{\x_i = \z_{ij} = \x_j, }{\quad}{\forall (i,j) \in \mathcal{E}}
    \addConstraint{\x_q = \btheta, \  }{\ }{\text{for one arbitrary } q \in [n],}
\end{mini}
where we introduce $\theta$ to separate the smooth and non-smooth functions in \eqref{eqn:DistObjDef}, and select (arbitrarily) the $q$-th agent to enforce the equality constraint as $\x_q = \btheta$.
By defining the source and destination matrices $\hat A_s, \hat A_d \in \R^{ |\mathcal{E}|\times n}$:
$[\hat A_s]_{ki} = [\hat A_d]_{kj} = 1$ if $(i,j)$ is the $k$-th element in $\mathcal{E}$ and all other entries are zero, stacking $\x_i, \z_{ij} \in \R^{d}$ into column vectors $\x \in \R^{nd}, \z \in \R^{ |\mathcal{E}|d}$, respectively,
denoting $f(\x) = \frac{1}{n}\sum_{i=1}^n f_i(\x_i)$, and defining $S := \mathbf{e}_q \otimes I_d \in \R^{nd \times d}$ ($\mathbf{e}_q \in \R^{n}$ has one at the $q$-th entry and zero elsewhere), we obtain a compact representation as:
\begin{mini}
    {\x, \z, \btheta }{f(\x) + r(\btheta),}
    {\label{eqn:DistObjConsensusMat}}{}
    \addConstraint{A\x = \begin{bmatrix} \hat A_s \otimes I_d\\ \hat A_d \otimes I_d \end{bmatrix} \x}{= \begin{bmatrix} I_{|\mathcal{E}|d}\\ I_{|\mathcal{E}|d} \end{bmatrix}\z = B\z}{}
    \addConstraint{S^{\top}\x }{= \btheta.}{}
\end{mini}
We further make the definitions $\hat E_u = \hat A_s + \hat A_d, \hat E_s = \hat A_s - \hat A_d \in \R^{ |\mathcal{E}|\times n}$ (unsigned and signed graph incidence matrix), $\hat L_u = {\hat E_u}^{\top}{\hat E_u}, \hat L_s = {\hat E_s}^{\top}{\hat E_s} \in \R^{ n\times n}$ (unsigned and signed graph Laplacian matrix), and $\hat D = \frac{1}{2}(\hat L_u + \hat L_s) \in \R^{ n\times n}$ (graph degree matrix with entries $D_{ii} = |\mathcal{N}_i|$).
Their block extensions are given by $E_u = \hat E_u \otimes I_d$, $E_s = \hat E_s \otimes I_d$, $L_u = \hat L_u \otimes I_d$, $L_s = \hat L_s \otimes I_d$, and $D = \hat D \otimes I_d$ (where $\otimes$ denotes the Kronecker product and $I_d$ the identity matrix). 

\begin{remark}
It is worth noting that the use of the intermediate variables $\{z_{ij}\}$ is essential for each agent to estimate curvature without information exchange in the neighborhood, leading to a communication-efficient protocol. 
For more details, we refer the reader to \cite[\textit{Remark} 2]{li2022DRUID}.
\end{remark}


\subsection{ADMM Formulation}
\label{subsec:ADMM}
The augmented Lagrangian (AL) for \eqref{eqn:DistObjConsensusMat} is
\begin{align}
        &\mathcal{L}(\x, \btheta, \z; \y, \blambda) :=  f(\x) + r(\btheta) + \y^{\top}(A\x - B\z) \label{eqn:DistObjConsensusMatAL}\\
        & + \blambda^{\top}(S^{\top}\x - \btheta) + \frac{\mu_{z}}{2}\|A\x - B\z\|^2 + \frac{\mu_{\theta}}{2}\|S^{\top}\x - \btheta\|^2. \notag
\end{align}
We apply (3-step) ADMM to \eqref{eqn:DistObjConsensusMatAL}:
\begin{subequations}
    \label{eqn:ADMM}
    \begin{align}
        \x^{t+1} &= \argmin_{\x} \mathcal{L}_{\Gamma}(\x, \btheta^t, \z^t; \y^t, \blambda^t), \label{eqn:ADMM_x}\\
        \btheta^{t+1} &= \argmin_{\btheta} \mathcal{L}(\x^{t+1}, \btheta, \z^t; \y^t, \blambda^t), \label{eqn:ADMM_theta}\\
        \z^{t+1} &= \argmin_{\z} \mathcal{L}(\x^{t+1}, \btheta^{t+1}, \z; \y^t, \blambda^t), \label{eqn:ADMM_z}\\
        \y^{t+1} &= \y^t + \mu_{z}(A \x^{t+1} - B \z^{t+1}), \label{eqn:ADMM_y}\\
        \blambda^{t+1} &= \blambda^t + \mu_{\theta}(S^{\top}\x^{t+1} - \btheta^{t+1}). \label{eqn:ADMM_lambda}
    \end{align}
\end{subequations}
Note that minimization in \eqref{eqn:ADMM_x} is carried over the perturbed AL:
\begin{equation}\label{eqn:perturbed_AL}
    \mathcal{L}_{\Gamma}(\x, \btheta^t, \z^t; \y^t, \blambda^t):= \mathcal{L}(\x, \btheta^t, \z^t; \y^t, \blambda^t) + \frac{1}{2}\|\x - \x^t\|_{\Gamma}^2, 
\end{equation}
where we define the coefficient matrix $\Gamma := \textbf{diag}\left(\epsilon_1,\hdots,\epsilon_n\right)\otimes I_d \in \R^{nd \times nd}$, and $\epsilon_i > 0, \forall i \in [n]$.
This serves to add robustness by means of \emph{personalized} hyperparameters ($\epsilon_i$ for each agent $i$).
The main computational burden resides in the sub-optimization problem \eqref{eqn:ADMM_x}, for which obtaining an exact solution would be  time-consuming \cite{shi2014linear};
to remedy this issue, vanilla DRUID \cite{li2022DRUID} adopts a single step of the same deterministic algorithm (GD, Newton, or BFGS) across all agents to solve \eqref{eqn:ADMM_x} inexactly. 

Our work serves to bridge the gap between distributed inexact \cite{li2022DRUID} and exact ADMM \cite{shi2014linear} by deviating from single-step updates and permitting a variable, agent-dependent number of updates ($E_i$ for agent $i$). 
This is necessary in order to further exploit local computation resources to attain a better trade-off between computation (more work) and communication (fewer global rounds).
For the local solver, we opt for a stochastic Newton method \cite{roosta2019sub} that employs sub-sampled Hessian inverse and stochastic gradient.
This is in view of the special structure as in \eqref{eqn:localloss}, for which stochastic methods have proliferated in problems with \emph{Big Data} (these methods operate by sampling mini-batches of data, and use tools such as gradient tracking and variance reduction to recover linear convergence \cite{roosta2019sub,johnson2013SVRG, Agarwal2017LiSSA}).

The following extends \cite[\textit{Lemma} 1]{li2022DRUID} to the case of personalized hyperparameters and variable local iterations of stochastic Newton, and is key for an efficient implementation of \eqref{eqn:ADMM} under agent-specific stochastic Newton steps for sub-problem \eqref{eqn:ADMM_x}.
Note that below we use a generic local iteration number $E$ for notational simplification but would like to emphasize that the implementation of the algorithm supports freely choosing the local work as $E_i$.
\begin{lemma}\label{lem:centralizeUpdates}
Consider \eqref{eqn:ADMM} with \eqref{eqn:ADMM_x} replaced by $E$ steps of stochastic Newton ($\hat \nabla,\hat{\nabla}^2$ denote stochastic gradient/Hessian). 
Let $\y^t = [\balpha^t;\bbeta^t]$, $\balpha^t,\bbeta^t \in \R^{|\mathcal{E}|d}$.
If $\y^0$ and $\z^0$ are initialized so that $\balpha^0 + \bbeta^0 = 0$ and $\z^0 = \frac{1}{2}E_u \x^0$, then 
$\balpha^t + \bbeta^t = 0$ and $\z^t = \frac{1}{2}E_u \x^t$ for all $t \ge 0$.
Moreover, defining $\bphi^t := E_s^{\top} \balpha^t$, $\x^{t,0} := \x^{t}$ and $\x^{t,E} := \x^{t+1}$, the updates can be written as:
\begin{subequations}\label{eqn:matrix_ADMM_update}
\begin{align}
    \x^{t,e} &= \x^{t,e-1} - \Big(\hat \nabla^2 f(\x^{t,e-1}) + \mu_z D + \mu_{\theta}SS^{\top}+\Gamma\Big)^{-1} \notag\\
    &\Big[\hat \nabla f(\x^{t,e-1})+\bphi^t+S\blambda^t+ \frac{1}{2} \mu_z L_s \x^{t} + \Gamma (\x^{t,e-1} - \x^{t})  \notag\\
    & +\mu_{\theta}S(S^{\top}\x^{t}-\btheta^t)\Big], \ e = 1, \hdots, E,\label{eqn:matrix_x}\\
    \btheta^{t+1} &= \textbf{prox}_{r/\mu_{\theta}}(S^{\top}\x^{t+1} + \frac{1}{\mu_{\theta} }\blambda^t),\label{eqn:matrix_theta}\\
    \bphi^{t+1} &= \bphi^{t} + \frac{1}{2} \mu_z L_s \x^{t+1},\label{eqn:matrix_phi}\\
    \blambda^{t+1} &= \blambda^t + \mu_{\theta} (S^{\top} \x^{t+1} - \btheta^{t+1}), \label{eqn:matrix_lambda}
\end{align}
\end{subequations}
where $\textbf{prox}$ denotes the proximal mapping \cite{parikh2014proximal}.
\end{lemma}
\begin{rproof}
    See Appendix.
\end{rproof}


\begin{algorithm}[htbp]
    \begin{algorithmic}[1]
    \Statex \textbf{Input:} global rounds $T$, work loads $E_i$
        \For{global round $t = 1, \hdots, T$}
            \For{each active agent $i$ \textbf{in parallel}}
\NOTE{Primal update:} 
                \State $\x_i^{t, 0} \leftarrow \x_i^{t}$
                \For{
                $e = 1, \hdots, E_i$}
                    \State sample batches $b_g$ and $b_H$
                    \State compute $\g_{i,b_g}^{t,e-1},H_{i,b_H}^{t,e-1}$ as in \eqref{eqn:sto_g_H}
                    \State solve $H_{i,b_H}^{t,e-1} \mathbf{d}_{i}^{t,e-1} = \g_{i,b_g}^{t,e-1}$
                    \State $\x_i^{t, e} \leftarrow \x_i^{t,e-1} - \mathbf{d}_{i}^{t,e-1}$
                \EndFor
                \State $\x_i^{t+1} \leftarrow \x_i^{t, E_i}$
\NOTE{Communication:}
                \State broadcast $\x_i^{t+1}$ to neighbors
\NOTE{Dual update:}
                \State  $\bphi_i^{t+1} \leftarrow \bphi_i^t + \frac{\mu_{z}}{2}\sum_{j\in \mathcal{N}_i}(\x_i^{t+1} - \x_j^{t+1})$
\NOTE{Updates pertaining to the regularizer:}
                \If{$i=q$}
                    \State $\btheta^{t+1} \leftarrow \textbf{prox}_{r/\mu_{\theta}}(\x_{q}^{t+1} + \frac{1}{\mu_{\theta} }\blambda^t)$
                    \State $\blambda^{t+1} \leftarrow \blambda^{t} + \mu_{\theta}(\x_{q}^{t+1} - \btheta^{t+1})$ 
                \EndIf

            \EndFor
            
        \EndFor
    \end{algorithmic}
\mycaptionof{algorithm}{\texttt{DRUID-VL}}\label{alg:NoName}
\end{algorithm}

\section{ALGORITHM}
\label{sec:algorithm}
The updates in \eqref{eqn:matrix_ADMM_update} can be directly written as a distributed protocol \eqref{eqn:dist_ADMM}, that is amenable to implementation with solely \emph{agent-based} variables. 
thIn specific, agent $i$ holds $(\x_i, \bphi_i)$ (the $q$-th agent associated with the regularizer $r(\cdot)$ additionally holds $(\btheta,\blambda)$) which are updated as: 
\begin{subequations}
    \label{eqn:dist_ADMM}
    \begin{align}
        \x_i^{t,e} &=  \x_i^{t,e-1} - \mathbf{d}_{i}^{t,e-1}, \ e = 1, \hdots, E_i, \label{eqn:dist_ADMM_x}\\
        \btheta^{t+1} &= \textbf{prox}_{r/\mu_{\theta}}(\x_{q}^{t+1} + \frac{1}{\mu_{\theta} }\blambda^t), \label{eqn:dist_ADMM_theta}\\
        \bphi_i^{t+1} &= \bphi_i^t + \frac{\mu_{z}}{2}\sum_{j\in \mathcal{N}_i}(\x_i^{t+1} - \x_j^{t+1}), \label{eqn:dist_ADMM_phi}\\
        \blambda^{t+1} &= \blambda^{t} + \mu_{\theta}(\x_{q}^{t+1} - \btheta^{t+1}). \label{eqn:dist_ADMM_lambda}
    \end{align}
\end{subequations}
Agent $i$ conducts $E_i$ local iterations as in \eqref{eqn:dist_ADMM_x}, where $\mathbf{d}_{i}^{t,e-1}$ is an approximated Newton step that is obtained as follows. 
In local iteration $e$, we sample mini-batches $b_g$ and $b_H$ from the local dataset of the $i$-th agent $\big\{\s_i^j\big\}_{j=1}^{D_i}$.
We then compute stochastic gradient $\g_{i,b_g}^{t,e-1}$ and sub-sampled Hessian $H^{t,e-1}_{i,b_H}$ for the local sub-problem with respect to $\x_i^{t,e-1}$ as follows: 
\begin{align}
    \g_{i,b_g}^{t,e-1} :=& \frac{1}{|b_g|}\sum_{j \in b_g}\nabla l(\s_i^j;\x_i^{t,e-1}) + \delta_{iq} \mu_{\theta}\left(\x_i^{t} - \btheta^t +\mu_{\theta}^{-1}\blambda^t\right) \notag \\
    +& \bphi_i^t + \frac{\mu_{z}}{2}\sum_{j \in \mathcal{N}_i}\left(\x_i^{t} - \x_j^t\right) + \epsilon_i\left(\x_i^{t,e-1} - \x_i^{t}\right), \label{eqn:sto_g_H}\\
    H^{t,e-1}_{i,b_H}:=&\frac{1}{|b_H|}\sum_{j \in b_H}\nabla^2 l(\s_i^j;\x_i^{t,e-1})+ (\mu_{z} |\mathcal{N}_i| + \delta_{iq} \mu_{\theta} + \epsilon_i)I_d, \notag
\end{align}
where $\delta_{iq} = 1$ if $i = q$ and $0$ otherwise.
By solving the linear system $H_{i,b_H}^{t,e-1} \mathbf{d}_{i}^{t,e-1} = \g_{i,b_g}^{t,e-1}$, we obtain the desired stochastic Newton step $\mathbf{d}_{i}^{t,e-1}$.

The details of our proposed method, DRUID-VL (VL: variable loads), are given in Algorithm \ref{alg:NoName}.
It admits asynchronous implementation by allowing variable number of participating agents in each global round (step 2). 
Each active agent executes four stages sequentially: 
(i) \emph{Primal update} (steps 3-10, using $E_i$ iterations of stochastic Newton), 
(ii) \emph{Communication} (broadcasting of the local variable $\x_i$ as in step 11), 
(iii) \emph{Dual update} (step 12) and 
(iv) \emph{Updates pertaining to the regularizer}  (steps 13-16, only for the $q-$th agent).

\section{CONVERGENCE ANALYSIS}
\label{sec:analysis}
We proceed to establish our main convergence theorem (\textit{Theorem} \ref{thm:main}) under the following assumptions.
The proofs of all theoretical results are deferred to the appendix.
\begin{assumption}
\label{assump:graph_connected}
The graph $\mathcal{G}$ is connected.
\end{assumption}
\begin{assumption}
\label{assump:convexity_continuous}
The mini-batches $\frac{1}{|b|}\sum_{j \in b}l(\s_i^j; \cdot)$ and the regularizer function $r(\cdot)$ satisfy the following conditions:

(i) For each agent $i\in[n]$, $\frac{1}{|b|}\sum_{j \in b}l(\s_i^j; \cdot)$ are twice continuously differentiable, $m_f$-strongly convex with $M_f$-Lipschitz continuous gradient, i.e., $\forall i \in [n], \x_i \in \R^d$,
\begin{equation}
    m_f I_d \preceq \frac{1}{|b|}\sum_{j \in b} \nabla^2 l(\s_i^j; \x_i) \preceq M_f I_d,
\end{equation}
where $0 < m_f \le M_f < \infty$.

(ii) The Hessians of $\frac{1}{|b|}\sum_{j \in b}l(\s_i^j; \cdot)$ are Lipschitz continuous with constant $L$, i.e., $\forall i \in [n], \ \x_i,\y_i \in \R^d$,
\begin{equation}
    \bigg\| \frac{1}{|b|}\sum_{j \in b} \nabla^2 l(\s_i^j; \x_i) - \frac{1}{|b|}\sum_{j \in b} \nabla^2 l(\s_i^j; \y_i) \bigg\| \le L \|\x_i - \y_i\|.
\end{equation}

(iii) $r(\cdot):\R^d \rightarrow \R$ is proper, closed, and convex, i.e., $\forall \x, \y \in \R^d$,
\begin{equation}
    (\x-\y)^{\top}(\partial r(\x) - \partial r(\y)) \ge 0,
\end{equation}
where the inequality is meant for arbitrary elements in the subdifferential.
\end{assumption}
\begin{assumption}\label{assump:participation}
The agent participation over rounds is independent with probability $p_i$ for agent $i$ and $p_{\min} := \min_{i\in[n]} p_i > 0$. 
\end{assumption}

In the following (\textit{Lemma}s \ref{lem:NewtonConverge}, \ref{lem:multiple}, and \textit{Theorem} \ref{thm:main}), expectation $\E$ denotes conditional expectation upon past mini-batch selections and agent activations (assuming independence of the two
processes).
It is taken over the mini-batch selection at the current step (\textit{Lemma}s \ref{lem:NewtonConverge}, \ref{lem:multiple}) and the user activation (\textit{Theorem} \ref{thm:main}).

The following lemma is a direct consequence of the analysis of stochastic Newton method \cite[\textit{Theorem} 11]{roosta2019sub} applied to our setting \eqref{eqn:dist_ADMM_x}.
\begin{lemma}\label{lem:NewtonConverge}
    Let $\x^{t,\star}$ be solution to \eqref{eqn:ADMM_x}.
    For sufficiently large batch sizes $|b_g|$ and $|b_H|$, there exists $c_i \in (0,1)$ for $i \in [n]$.
    For $C := \textbf{diag}\left(c_1^{E_1}, \hdots, c_n^{E_n}\right)\otimes I_{nd} \in \R^{nd\times nd}$, it holds that 
    \begin{equation}\label{eqn:NewtonConverge}
        \E\Big[\|\x^{t,E}-\x^{t,\star}\|^2\Big] \le \|\x^{t,0} - \x^{t,\star}\|^2_{C}.
    \end{equation}
\end{lemma}
The next lemma characterizes the residual error of primal sub-problem \eqref{eqn:ADMM_x} induced by inexact updates as in \eqref{eqn:dist_ADMM_x}.
Recall that $\y^t = [\balpha^t;-\balpha^t]$ and $\bphi^t := E_s^{\top} \balpha^t$.
\begin{lemma}\label{lem:error}
    The residual error is given by:
    \begin{align}
        \e^t &:= \nabla \mathcal{L}_{\Gamma}(\x^{t+1}, \btheta^t, \z^t; \y^t, \blambda^t) \label{eqn:residual}\\
            &= \nabla f(\x^{t+1}) + E_s^{\top}\balpha^{t+1} + \mu_z E_u^{\top}(\z^{t+1} - \z^t) \notag\\
            & + S(\blambda^{t+1} + \mu_{\theta}(\btheta^{t+1} - \btheta^t)) + \Gamma(\x^{t+1} - \x^t). \label{eqn:expand}
    \end{align}
\end{lemma}
Note that from \textit{Assumption} \ref{assump:convexity_continuous} it follows that $\mathcal{L}_{\Gamma}(\x, \btheta^t, \z^t; \y^t, \blambda^t)$ in \eqref{eqn:perturbed_AL} has Lipschitz continuous gradient with parameter
\begin{equation}\label{eqn:Lipschitz_perturbAL}
    M := M_f + \mu_z\max_{i\in[n]} |\mathcal{N}_i| + \mu_{\theta}  + \max_{i\in[n]}\epsilon_i.
\end{equation}
The following lemma provides an upper bound for the residual error given in \textit{Lemma} \ref{lem:error} in relation to the difference of two consecutive primal iterates $\x^{t+1}-\x^t$.
\begin{lemma}\label{lem:multiple}
Recall $C := \textbf{diag}\left(c_1^{E_1}, \hdots, c_n^{E_n}\right)\otimes I_{nd}$.
For any $i \in [n]$, under sufficiently large batch sizes there exists $c_i \in (0,1)$ so that for any $\xi_i \in \left(0, c_i^{-1}-1\right)$ it holds that:
\begin{equation}\label{eqn:error_bound}
    \E \left[ \|\e^t\|^2\right] \le \E\left[\|\x^{t+1}-\x^t\|^2_{\mathcal{T}(E)}\right], 
\end{equation}
where $\mathcal{T}(E):=\textbf{diag}\left(\tau_1(E_1), \hdots, \tau_n(E_n)\right)\otimes I_{nd} \in \R^{nd \times nd}$ with
\begin{equation}\label{eqn:tau_def}
    \tau_i(E_i):=\frac{\left(1+\xi_i^{-1}\right) M^2 c_i^{E_i}}{1 - (1+\xi_i)c_i^{E_i}}, \ i \in [n].
\end{equation}
\end{lemma}
In what follows, we present the main convergence theorem that establishes linear convergence in expectation of DRUID-VL.
\begin{theorem}\label{thm:main}
Define $P := \textbf{diag}(p_i)$ and
\begin{align}
    \mathcal{H} &:= \textbf{diag}\left(\Gamma, 2\mu_{z} I_{|\mathcal{E}|d}, 2\mu_z^{-1} I_{|\mathcal{E}|d}, \mu_{\theta}I_d, \mu_{\theta}^{-1}I_d\right), \label{eqn:def_H}\\
    \v &:= \left[\x;\z;\balpha;\btheta;\blambda\right] \in \R^{(n+2|\mathcal{E}|+2)d}. \label{eqn:def_v}
\end{align}
Let
\begin{equation}\label{eqn:eta_selection}
    \begin{aligned}
        \eta = \min &\Bigg\{\min_{i \in [n]}\frac{\mu_{\theta} \sigma^{+}_{\min}(\epsilon_i - \zeta\tau_i(E_i))}{5(\tau_i(E_i)+\epsilon_i^2)}, \\ &\left(\frac{2m_f M_f}{m_f+M_f} - \frac{1}{\zeta} \right) \frac{1}{\max_{i \in [n]}\epsilon_i+\mu_{\theta}(\sigma^{L_u}_{\max} + 2)} , \\
        &  \frac{2}{5}\frac{\mu_{\theta} \sigma^{+}_{\min}}{m_f + M_f}, \frac{\sigma^{+}_{\min}}{5 \max\left\{1, \sigma^{L_u}_{\max}\right\}}, \frac{1}{2}\Bigg\},
    \end{aligned}
\end{equation}
where $\sigma^+_{\min}$ denotes the smallest positive eigenvalue of $\begin{bmatrix} E_s \\ S^{\top} \end{bmatrix} \begin{bmatrix} E_s \, S^{\top} \end{bmatrix}$, $\sigma^{L_u}_{\max}$ denotes the maximum eigenvalue of $L_u$, $\zeta \in \left(\frac{m_f + M_f}{2m_f M_f}, \min_{i \in [n]}\frac{\epsilon_i}{\tau_i(E_i)}\right)$, and $\epsilon_i > \frac{\tau_i(E_i)(m_f + M_f)}{2m_f M_f}, i \in [n]$,
then it holds that: 
\begin{equation}\label{eqn:Theorem_result}
    \E\left[\|\v^{t+1} - \v^{\star}\|^2_{\mathcal{H}P^{-1}}\right] \le \left(1-\frac{\eta p_{\min}}{1+\eta}\right)\|\v^{t} - \v^{\star}\|^2_{\mathcal{H}P^{-1}}.
\end{equation}
\end{theorem}

\begin{figure*}[!ht]
\centering
\begin{subfigure}[t]{0.4\textwidth}
\includegraphics[width=\columnwidth]{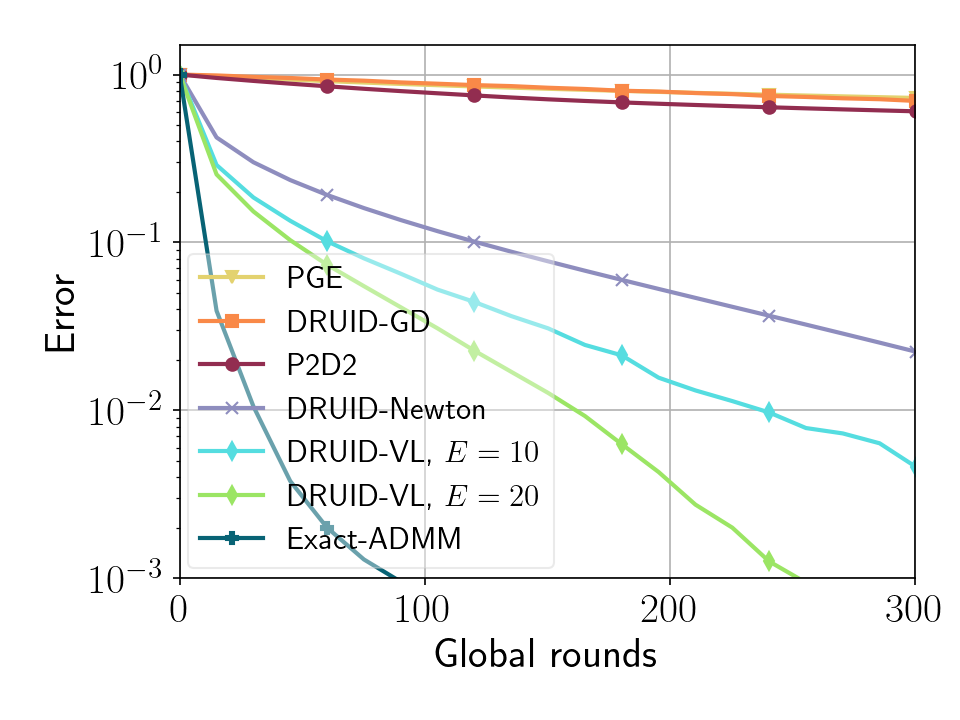}
\caption{Relative error versus global rounds.}\label{fig:convergence}
\end{subfigure}
\begin{subfigure}[t]{0.4\textwidth}
\includegraphics[width=\columnwidth]{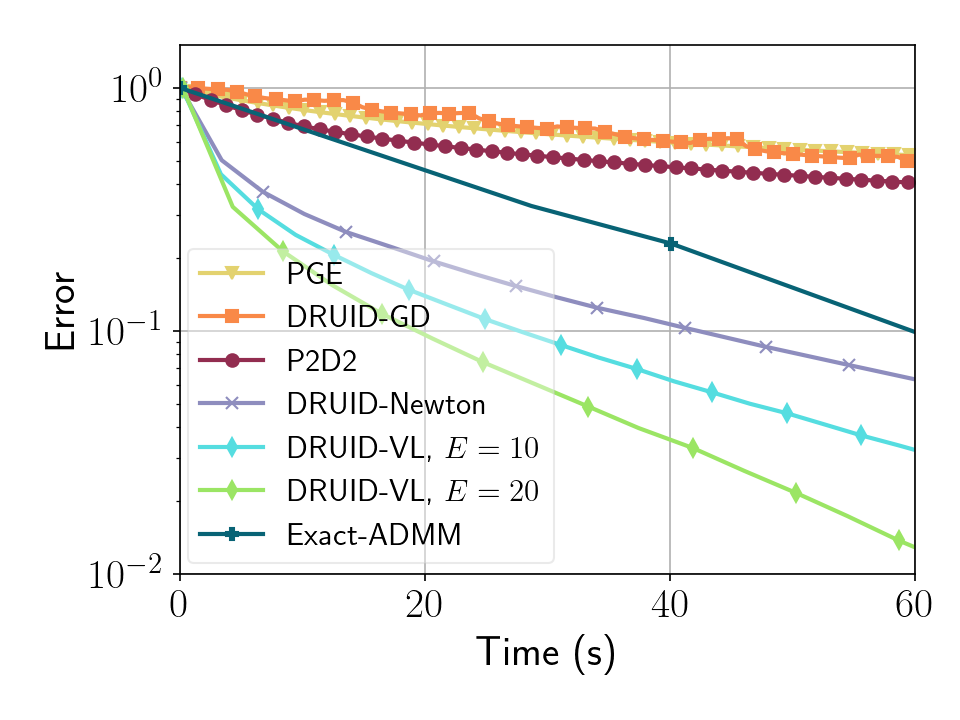}
\caption{Relative error versus time.}\label{fig:convergence_time}
\end{subfigure}
\caption{Comparisons of DRUID-VL (synchronous setting and common $E$ for all agents) with baseline algorithms in terms of relative error $\frac{\|\x^t - \x^{\star}\|^2}{\|\x^0 - \x^{\star}\|^2}$; convergence versus (a) global rounds and (b) run-time.}
\label{fig:all}
\end{figure*}


\begin{remark}
ADMM is a first-order method and thus converges linearly (the best possible rate \cite{shi2014linear}). 
Our goal is to accelerate convergence inside the ADMM framework by performing variable work loads on agents and therefore the result in \textit{Theorem} \ref{thm:main} falls into the taxonomy of linear convergence.
On the other hand, superlinear convergence to the global optimum in distributed networks requires heavy communication costs among agents (usually multiple message exchanges in one round \cite{wei2013distributed, mokhtari2016network, crane2019dingo}, and does not admit asynchronous implementation).
In contrast, our proposed method requires a single communication step of size $d$ and is asynchronous, suitable for real-life multi-agent networks where communication is the main bottleneck.
\end{remark}

In the following corollary, we show that the personalized coefficient $\epsilon_i$ (the diagonal entries of $\Gamma$ in \eqref{eqn:perturbed_AL}) can be further optimized to adjust to the heterogeneity of variable work loads, so as to achieve faster convergence (i.e., increase $\eta$ in \eqref{eqn:eta_selection}).
\begin{corollary}\label{lem:tune_epsilon}
By selecting
\begin{equation}\label{eqn:epsilon_selection}
    \epsilon_i = \zeta \tau_i(E_i) + \sqrt{\zeta^2 \tau_i^2(E_i) + \tau_i(E_i)},
\end{equation}
for $\zeta > \frac{m_f + M_f}{2m_f M_f}$, \eqref{eqn:Theorem_result} holds with
\begin{equation}\label{eqn:new_eta_selection}
    \begin{aligned}
        \eta = \min &\Bigg\{\tfrac{\mu_{\theta} \sigma^{+}_{\min}}{10\zeta}\min_{i \in [n]} \left(\sqrt{\tau^2_i(E_i) + \zeta^{-2}\tau_i(E_i)} + \tau_i(E_i) \right)^{-1}, \\ &\left(\frac{2m_f M_f}{m_f+M_f} - \frac{1}{\zeta} \right) \frac{1}{\max_{i \in [n]}\epsilon_i+\mu_{\theta}(\sigma^{L_u}_{\max} + 2)} , \\
        & \frac{2}{5}\frac{\mu_{\theta} \sigma^{+}_{\min}}{m_f + M_f}, \frac{\sigma^{+}_{\min}}{5 \max\left\{1, \sigma^{L_u}_{\max}\right\}}, \frac{1}{2}\Bigg\}.
    \end{aligned}
\end{equation}
\end{corollary}


\section{EXPERIMENTS}
We consider distributed sparse logistic regression, i.e, 
$$l(\s;\hat \x):= \log\left(1+e^{\hat \x^{\top} \mathbf{w}} \right) +(1-y)\hat \x^{\top} \mathbf{w},$$  
for $\s=(\mathbf{w},y)\in\mathbb{R}^d\times\{0,1\}$ and $r(\hat x) = \gamma \|\hat x\|_1$, with $\gamma = 2\times10^{-6}$.
We take $4,000$ samples of dimension $d=22$ from \emph{ijcnn1}\footnote{Available at https://www.csie.ntu.edu.tw/\textasciitilde cjlin/libsvm/.} and evenly distribute them to $10$ agents, whose network topology is generated by the Erd\H{o}s-R\'{e}nyi model with $p=0.2$ (each edge is generated independently with probability $0.2$).
We first evaluate our proposed method, DRUID-VL (with the same work load $E=10$ and $20$ across all agents), against five baselines, namely PGE \cite{shi2015extra}, P2D2 \cite{alghunaim2019P2D2}, DRUID-GD/-Newton \cite{li2022DRUID}, and Exact-ADMM (i.e., solving \eqref{eqn:ADMM_x} `exactly', which means assuring a primal-dual gap of $10^{-5}$, in \emph{Primal update} of Algorithm \ref{alg:NoName}, while other stages are unchanged), considering only synchronous updates as all but DRUID do not support asynchronous implementation.
We set $\epsilon=\mu_{\theta} = 2 \mu_{z} = 10^{-4}$ (the selection follows from our analysis in \eqref{eqn:upper_bound_first2}), and sampling batch size $|b_g| = |b_H| = 100$ for DRUID-VL.

The convergence of each algorithm is plotted in Fig.~\ref{fig:convergence} and \ref{fig:convergence_time}, in terms of global rounds and time, respectively.
Fig.~\ref{fig:convergence} demonstrates that for a fixed target error, DRUID-VL with $E=10/20$ requires $41\%/59\%$ less global rounds compared to DRUID-Newton, while the benefits are much higher for the other (gradient-based) baselines.
Fig.~\ref{fig:convergence_time} illustrates the progress of each algorithm per actual run-time (this is done since the cost of one round using a second-order local solver is higher than for gradient-based methods). 
We deduce that for a given time, DRUID-VL with $E=10/20$ decreases the relative error by $1.96\times/4.76\times$ of DRUID-Newton, and $12.5\times/33.3\times$ of PGE, DRUID-GD, and P2D2.
We also emphasize that although Exact-ADMM attains the fastest convergence in terms of global rounds (Fig.~\ref{fig:convergence}), solving the sub-problem exactly brings about heavy computational overhead. 
As a result,  Exact-ADMM has the slowest convergence speed among all DRUID variants in real-time scale (Fig.~\ref{fig:convergence_time}).
These findings corroborate that our proposed method better exploits local computation resources to reduce communication costs and run-time.

We further study the impact of heterogeneity and participation rate in Fig.~\ref{fig:heterogenity}, in terms of the communication cost to reach a given accuracy ($10^{-2}$ in all cases). 
The former is achieved by three alternatives for work load selection (for fair comparison we set the average load fixed): 
\emph{Equal} - $E_i = 10$ for all $i$; 
\emph{Uniform} - $\left\{E_i\right\}_{i=1}^n$ are sampled at uniformly random from $\sim\mathcal{U}\{1,19\}$; 
\emph{Extreme} - half of agents are assigned $E_i = 1$ while the others have $E_i = 19$.
We use the same parameter setting as in Fig.~\ref{fig:all} and fix $\epsilon_i = 10^{-4}$ for all agents.
Fig.~\ref{fig:heterogenity} shows that the participation rate $p_{\min}$ does not have a significant impact on the communication burden per agent, in full accord with \eqref{eqn:Theorem_result}. 
More interestingly, heterogeneity is shown to play a major role in exacerbating communication costs:
higher variance of work load $\left\{E_i\right\}_{i=1}^n$ (e.g., in our case, \emph{Extreme} selection) yields significant rise in communication cost.

\begin{figure}[htbp]
    \centering
    \includegraphics[width=0.3\textwidth]{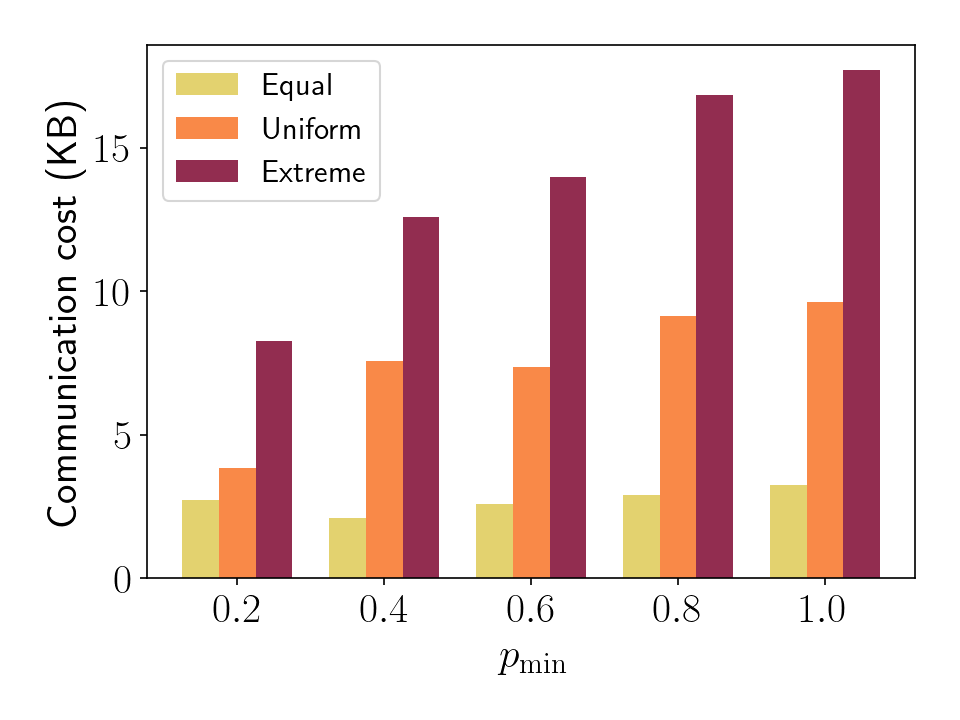}
    \caption{Average communication cost per agent to reach an error of $10^{-2}$ under variable heterogeneity and participation rate.}
    \label{fig:heterogenity}
\end{figure}

Based on the analysis in \textit{Corollary} \ref{lem:tune_epsilon}, we further tune the local hyperparameters $\epsilon_i$ so as to adjust to heterogeneity. 
Given that the expression in \eqref{eqn:epsilon_selection} cannot be analytically evaluated in a real scenario (where constants $M_f, m_f, c_i$ are unknown), we choose a surrogate rule based on \eqref{eqn:tau_def} and \eqref{eqn:epsilon_selection} as:
$$\epsilon_i = \bar{\epsilon}c^{E_i - \bar{E}}[1-(1+\zeta)c^{\bar{E}}]/[1-(1+\zeta)c^{E_i}],$$ 
where we select $\bar{\epsilon}=10^{-4}$, $\bar{E}=10$, $c=0.98$, $\zeta=5\times10^{-3}$.
The results are plotted in Fig.~\ref{fig:optimize_epsilon}, for the \emph{Uniform} heterogeneity setting (as in Fig.~\ref{fig:heterogenity}).
We observe a noticeable reduction of the number of global rounds needed to reach a common accuracy ($10^{-2}$).
Maximum gain from hyperparameter tuning is achieved at $p_{\min}=0.4$ ($28.2\%$ less rounds).
Last but not least, the gain increases when the participation rate $p_{\min}$ decreases, indicating that our proposed method is well-suited to adapt to heterogeneity in the most challenging (albeit common in real-systems) scenario of `high' asynchrony.

\begin{figure}[htbp]
    \centering
    \includegraphics[width=0.42\textwidth]{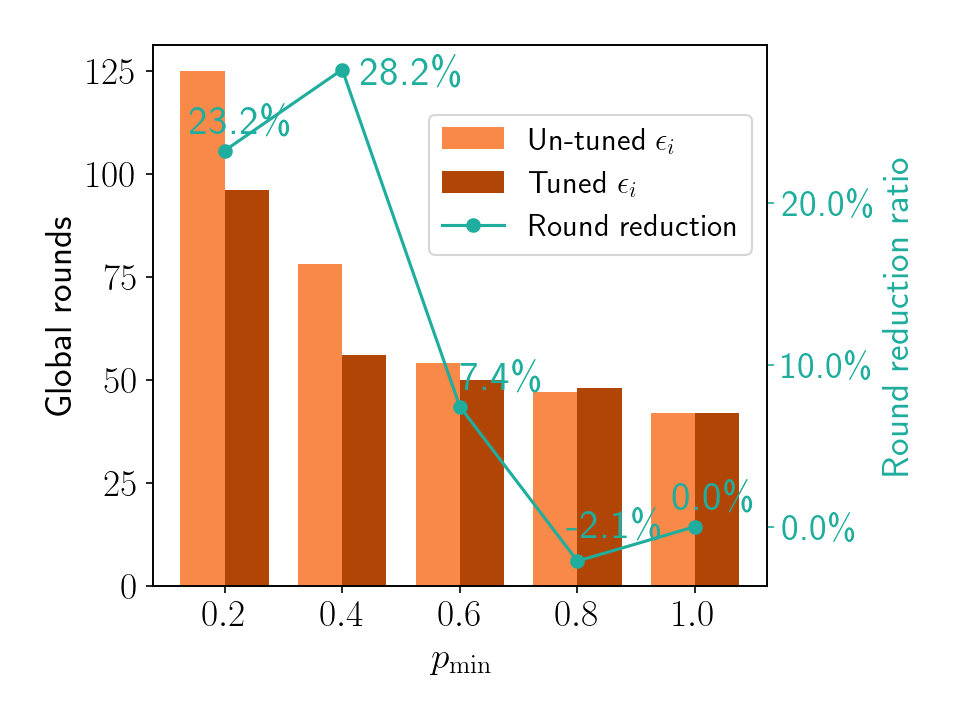}
    \caption{Global rounds required to reach an error of $10^{-2}$ with tuned and un-tuned $\epsilon_i$ under \emph{Uniform} selection of $\left\{E_i\right\}_{i=1}^n$.}
    \label{fig:optimize_epsilon}
\end{figure}

\section{CONCLUSION}
We proposed a distributed asynchronous primal-dual method tailored for heterogeneous multi-agent systems (in terms of variability of local work). 
The algorithm allows for a variable number of participating agents and requires a single broadcast of the local vector. 
Linear convergence was established for arbitrary degree of heterogeneity (\textit{Theorem}~\ref{thm:main}). 
Additionally, a simple method of personalizing hyperparameters locally (\textit{Corollary}~\ref{lem:tune_epsilon}) was shown to accelerate the algorithm both theoretically and experimentally.



\bibliographystyle{IEEEbib}
\bibliography{refs}

\clearpage
\appendix
\begin{proof}[Proof of \textbf{\textit{Lemma}} \ref{lem:centralizeUpdates}]
    The decomposition of $\y^t = [\balpha^t;-\balpha^t]$ and the property that $\z^t = \frac{1}{2}E_u \x^t$ follow the same lines as in the proof of \cite[\textit{Lemma} 1]{li2022DRUID}.
    \eqref{eqn:matrix_x} is obtained by using multiple stochastic Newton steps for \eqref{eqn:ADMM_x} and invoking $\z^t = \frac{1}{2}E_u \x^t$ (which eliminates \eqref{eqn:ADMM_z}) and the property that $\y^t = [\balpha^t;-\balpha^t]$;
    by completion of squares, \eqref{eqn:ADMM_theta} can be converted to \eqref{eqn:matrix_theta} under \textit{Assumption} \ref{assump:convexity_continuous}.
    \eqref{eqn:matrix_phi} is equivalent to \eqref{eqn:ADMM_y}, achieved by inspecting the symmetric structure in $\y$ and the proper initialization $\balpha^0 + \bbeta^0 = 0$;
    finally, \eqref{eqn:matrix_lambda} is equivalent to \eqref{eqn:ADMM_lambda}.
\end{proof}
\begin{proof}[Proof of \textbf{\textit{Lemma}} \ref{lem:NewtonConverge}]
    It follows from the analysis of stochastic Newton method \cite[\textit{Theorem} 11]{roosta2019sub} that for sufficiently large batch sizes $|b_g|$ and $|b_H|$, there exist $c_i \in (0,1)$ such that for $i \in [n]$:
    \begin{equation*}\label{eqn:local_sto_Newton}
        \E\left[\|\x^{t,E_i}_i-\x^{t,\star}_i\|^2\right] \le c_i^{E_i}\|\x^{t,0}_i - \x^{t,\star}_i\|^2,
    \end{equation*}
    whence \eqref{eqn:NewtonConverge} follows from the definition of $C$.
\end{proof}
\begin{proof}[Proof of \textbf{\textit{Lemma}} \ref{lem:error}]
    The gradient of the perturbed AL in \eqref{eqn:perturbed_AL} at global round $t$ is given by:
    \begin{align}
        &\nabla \mathcal{L}_{\Gamma}(\x, \btheta^t, \z^t; \y^t, \blambda^t) := \nabla f(\x) + A^{\top}\y^{t} + S\blambda^{t} \label{eqn:grad_epsAL} \\
        &+ \mu_{z}A^{\top}\left(A\x - B\z^{t}\right) + \mu_{\theta}S(S^{\top}\x - \btheta^t) + \Gamma(\x - \x^t). \notag
    \end{align}
    Following from \textit{Lemma} \ref{lem:centralizeUpdates}, we obtain: 
    \begin{equation*}
        \begin{aligned}
            S\blambda^{t} + \mu_{\theta}S(S^{\top}\x^{t+1} - \btheta^t) &= S\blambda^{t+1} + \mu_{\theta}S(\btheta^{t+1} - \btheta^t),\\
            A^{\top}\y^{t} + \mu_{z}A^{\top}\left(A\x^{t+1} - B\z^{t}\right) &= \bphi^t + \frac{\mu_z}{2}\left(2D\x^{t+1} - L_u \x^t\right),\\
            \bphi^t + \frac{\mu_z}{2}\left(2D\x^{t+1} - L_u \x^t\right) &= \bphi^{t+1} + \frac{\mu_z}{2} L_u\left(\x^{t+1} - \x^t\right),\\
            \frac{\mu_z}{2} L_u\left(\x^{t+1} - \x^t\right) &= \mu_z E_u^{\top}(\z^{t+1} - \z^t),
        \end{aligned}
    \end{equation*}
    in specific, the first equality follows from \eqref{eqn:matrix_lambda} and $S^{\top}\x^{t+1} = \btheta^{t+1}$;
    the second from \eqref{eqn:ADMM_y}, \eqref{eqn:matrix_phi}, and $\bphi^t = E_s^{\top} \balpha^t$;
    the third from $D = \frac{1}{2}(L_u + L_s)$;
    finally, the last from $\z^t = \frac{1}{2}E_u \x^t$.
    Substituting them into \eqref{eqn:grad_epsAL} with $\x = \x^{t+1}$, we obtain \eqref{eqn:residual}.
\end{proof}
\begin{proof}[Proof of \textbf{\textit{Lemma}} \ref{lem:multiple}]
The Lipschitz continuity of $\nabla \mathcal{L}_{\Gamma}(\x, \btheta^t, \z^t; \y^t, \blambda^t)$ with parameter $M$ as in \eqref{eqn:Lipschitz_perturbAL} is a direct consequence of \textit{Assumption} \ref{assump:convexity_continuous} and the
definition in \eqref{eqn:perturbed_AL}.
It then follows from \textit{Lemma} \ref{lem:error} that
\begin{equation}\label{eqn:multiple_error_bound}
\begin{aligned}
    \|\e^t\|^2 &=\|\nabla \mathcal{L}_{\Gamma}(\x^{t+1}, \btheta^t, \z^t; \y^t, \blambda^t)\|^2\\
    &\le M^2 \|\x^{t+1}-\x^{t,\star}\|^2,
\end{aligned}
\end{equation}
From the inequality $\|a+b\|^2 \le (1+\xi^{-1})\|a\|^2 + (1+\xi)\|b\|^2$ for any $\xi>0$, we obtain for any $i\in[n]$
\begin{equation}\label{eqn:multiple_intermediate_inequa}
    \begin{aligned}
        \|\x^{t}_i-\x^{t,\star}_i\|^2 &\le (1+\xi_i^{-1})\|\x^{t}_i-\x^{t+1}_i\|^2\\
        &+ (1+\xi_i)\|\x^{t+1}_i-\x^{t,\star}_i\|^2.
    \end{aligned}
\end{equation}
Taking expectation on both sides of \eqref{eqn:multiple_intermediate_inequa} and invoking \textit{Lemma} \ref{lem:NewtonConverge} (recall $\x^{t+1}_i := \x^{t,E_i}_i$), we obtain
\begin{equation}\label{eqn:multiple_final_inequa}
    \begin{aligned}
        \E\left[\|\x^{t+1}_i-\x^{t,\star}_i\|^2\right] &\le c_i^{E_i}\E\left[\|\x^{t}_i-\x^{t,\star}_i\|^2\right]\\ 
        &\le (1+\xi_i^{-1})c_i^{E_i}\E\left[\|\x^{t+1}_i-\x^{t}_i\|^2\right] \\
        &+(1+\xi_i)c_i^{E_i}\E\left[\|\x^{t+1}_i-\x^{t,\star}_i\|^2\right].
    \end{aligned}
\end{equation}
Rearranging the terms in \eqref{eqn:multiple_final_inequa}, using \eqref{eqn:multiple_error_bound}, and combining the inequalities for each $i\in[n]$ into a single quadratic form gives the desired.
\end{proof}
\begin{proof}[Proof of \textbf{\textit{Theorem}} \ref{thm:main}]
We first analyze the synchronous method, i.e., $p_{\min}=1$.
Combining \eqref{eqn:expand} and KKT condition $\nabla f(\x^{\star}) + E_s^{\top}\balpha^{\star} + S\blambda^{\star} = 0$ gives:
\begin{equation}\label{eqn:originDEF}
    \begin{aligned}
    \nabla f(\x^{t+1})-\nabla f(\x^\star) = -\big(E_s^\top (\balpha^{t+1}-\balpha^\star)+\Gamma(\x^{t+1}-\x^t)\\
    +S(\blambda^{t+1}-\blambda^\star+\mu_{\theta} (\btheta^{t+1}-\btheta^t))-\e^t+\mu_{z} E_u^\top(\z^{t+1}-\z^t)).
    \end{aligned}
\end{equation}
Under \textit{Assumption} \ref{assump:convexity_continuous}, the following inequality holds \cite[\textit{Theorem} 2.1.12]{nesterov2018lectures}:
\begin{equation}\label{eqn:main_1}
    \begin{aligned}
        &\tfrac{m_fM_f}{m_f+M_f}\|\x^{t+1}-\x^\star\|^2+\tfrac{1}{m_f+M_f}\|\nabla f(\x^{t+1})-\nabla f(\x^\star)\|^2 \\
        &\le (\x^{t+1}-\x^\star)^\top (\nabla f(\x^{t+1})-\nabla f(\x^\star)) \\
        &= (\x^{t+1}-\x^\star)^\top \e^t-(\x^{t+1}-\x^\star)^\top\Gamma(\x^{t+1}-\x^t)\\
        &-(\x^{t+1}-\x^\star)^{\top} E_s^{\top} (\balpha^{t+1}-\balpha^\star)\\
        &-(\x^{t+1}-\x^\star)^\top S \Big(\blambda^{t+1}-\blambda^\star +\mu_{\theta}(\btheta^{t+1}-\btheta^t)\Big)\\
        &-\mu_{z}(\x^{t+1}-\x^\star)^\top E_u^\top (\z^{t+1}-\z^t). 
    \end{aligned}
\end{equation}
The following equalities follow from \textit{Lemma} \ref{lem:centralizeUpdates}.
In specific, the first is from \eqref{eqn:matrix_phi} and KKT condition $E_s \x^{\star} = 0$;
the second is from \eqref{eqn:matrix_theta} and KKT condition $S^{\top} \x^{\star} = \btheta^{\star}$;
the third follows from $\z^t = \frac{1}{2}E_u \x^t$.
The following holds
\begin{align*}
    (\x^{t+1}-\x^\star)^{\top} E_s^{\top} &= \frac{2}{\mu_z}(\balpha^{t+1}-\balpha^{t}), \\
    (\x^{t+1}-\x^\star)^\top S &= (\btheta^{t+1}-\btheta^{\star})^{\top} + \frac{1}{\mu_{\theta}}(\blambda^{t+1}-\blambda^{t})^{\top}, \\
    (\x^{t+1}-\x^\star)^\top E_u^\top &= 2(\z^{t+1}-\z^{\star})^{\top}.
\end{align*}
Using the identity $-2(a-b)^{\top}\mathcal{K}(a-c)=\|b-c\|^2_{\mathcal{K}}-\|a-b\|^2_{\mathcal{K}}-\|a-c\|^2_{\mathcal{K}}$ for $\mathcal{K} \succ 0$ to \eqref{eqn:main_1} (for $\mathcal{K}=\Gamma$ or $I$), we obtain
\begin{align}
    &\tfrac{2m_fM_f}{m_f+M_f}\|\x^{t+1}-\x^\star\|^2+\tfrac{2}{m_f+M_f}\|\nabla f(\x^{t+1})-\nabla f(\x^\star)\|^2 \notag\\
    &\leq \big(\|\x^t-\x^\star\|^2_{\Gamma}-\|\x^{t+1}-\x^\star\|^2_{\Gamma}-\|\x^{t+1}-\x^t\|^2_{\Gamma}\big) \notag\\
    & +2\mu_{z}\big(\|\z^t-\z^\star\|^2-\|\z^{t+1}-\z^\star\|^2-\|\z^{t+1}-\z^t\|^2\big) \notag\\
    & +\frac{1}{\mu_\theta}\big(\|\blambda^t-\blambda^\star\|^2-\|\blambda^{t+1}-\blambda^\star\|^2-\|\blambda^{t+1}-\blambda^t\|^2\big) \notag\\
    & +\mu_{\theta}\big(\|\btheta^t-\btheta^\star\|^2-\|\btheta^{t+1}-\btheta^\star\|^2-\|\btheta^{t+1}-\btheta^t\|^2\big) \label{eqn:bound_various_terms}\\
    & +\frac{2}{\mu_{z}}\big(\|\balpha^{t}-\balpha^\star\|^2-\|\balpha^{t+1}-\balpha^\star\|^2-\|\balpha^{t+1}-\balpha^t\|^2\big) \notag\\
    & +2(\x^{t+1}-\x^\star)^\top \e^t \notag\\
    &=\|\v^t-\v^{\star}\|^2_{\mathcal{H}}-\|\v^{t+1}-\v^{\star}\|^2_\mathcal{H}-\|\v^{t+1}-\v^t\|^2_{\mathcal{H}} \notag\\
    & +2(\x^{t+1}-\x^\star)^\top \e^t. \notag
\end{align}
Establishing linear convergence boils down to showing that 
\begin{equation}\label{eqn:thm_general_inequa}
    \eta \|\v^{t+1}-\v^{\star}\|^2_{\mathcal{H}} \leq \|\v^t-\v^{\star}\|^2_{\mathcal{H}}-\|\v^{t+1}-\v^{\star}\|^2_{\mathcal{H}}
\end{equation}
holds for some $\eta>0$.
From the definitions of $\mathcal{H}$ and $\v$ in \eqref{eqn:def_H} and \eqref{eqn:def_v}, the LHS of \eqref{eqn:thm_general_inequa} can be expanded to 
\begin{equation}\label{eqn:first_term_expansion}
    \begin{aligned}
        &\eta \|\v^{t+1}-\v^{\star}\|^2_{\mathcal{H}} = \eta \Big(\|\x^{t+1}-\x^\star\|^2_{\Gamma} + 2\mu_{z}\|\z^{t+1}-\z^\star\|^2\\
        & + \mu_{\theta}\|\btheta^{t+1}-\btheta^\star\|^2 + \tfrac{2}{\mu_{z}} \|\balpha^{t+1}-\balpha^\star\|^2 +  \tfrac{1}{\mu_{\theta}}\|\blambda^{t+1}-\blambda^\star\|^2\Big).
    \end{aligned}
\end{equation}
In the following, we will give an upper bound for \eqref{eqn:first_term_expansion}.
First, we show an intermediate inequality:
\begin{equation}\label{eqn:bound_alpha_lambda}
    \begin{aligned}
        &\|\balpha^{t+1}-\balpha^\star\|^2 + \|\blambda^{t+1}-\blambda^\star\|^2 \\
        &\le \frac{1}{\sigma^+_{\min}}\|E_s^{\top}(\balpha^{t+1}-\balpha^\star) + S(\blambda^{t+1}-\blambda^\star)\|^2 \\
        &\le \frac{5}{\sigma^+_{\min}} \Big(\|\nabla f(\x^{t+1}) - \nabla f(\x^{\star})\|^2 + \|\x^{t+1}-\x^t\|^2_{\Gamma^2}\\
        &+\mu_{\theta}^2\|\btheta^{t+1} -\btheta^{t}\|^2+\|\e^{t}\|^2+\sigma^{L_u}_{\max}\mu_{z}^2\|\z^{t+1}-\z^t\|^2\Big),
    \end{aligned}
\end{equation}
where the first inequality follows from \cite[\textit{Lemma} 5]{li2022DRUID} and the second follows from \eqref{eqn:originDEF} as well as $(\sum_{i=1}^n a_i)^2 \le \sum_{i=1}^n na_i^2$.
By selecting $\mu_z = 2 \mu_{\theta}$, we obtain from \eqref{eqn:bound_alpha_lambda} that
\begin{equation}\label{eqn:upper_bound_first2}
    \begin{aligned}
        &\frac{2}{\mu_{z}} \|\balpha^{t+1}-\balpha^\star\|^2 + \frac{1}{\mu_{\theta}}\|\blambda^{t+1}-\blambda^\star\|^2\\
        &=\frac{1}{\mu_{\theta}} \big(\|\balpha^{t+1}-\balpha^\star\|^2 + \|\blambda^{t+1}-\blambda^\star\|^2 \big)\\
        &\le \frac{5}{\mu_{\theta} \sigma^+_{\min}} \Big(\|\nabla f(\x^{t+1}) - \nabla f(\x^{\star})\|^2 + \|\x^{t+1}-\x^t\|^2_{\Gamma^2}\\
        &+\mu_{\theta}^2\|\btheta^{t+1} -\btheta^{t}\|^2+\|\e^{t}\|^2+\sigma^{L_u}_{\max}\mu_{z}^2\|\z^{t+1}-\z^t\|^2\Big),
    \end{aligned}
\end{equation}
which provides the upper bound for the last two terms on the RHS of \eqref{eqn:first_term_expansion}.
From \textit{Lemma} \ref{lem:centralizeUpdates}, the remaining terms in \eqref{eqn:first_term_expansion} can be upper bounded as:
\begin{equation}\label{eqn:upper_bound_second2}
    \begin{aligned}
        &2\mu_{z}\|\z^{t+1}-\z^\star\|^2 + \mu_{\theta}\|\btheta^{t+1}-\btheta^\star\|^2 \le\\
        &\tfrac{\mu_z \sigma^{L_u}_{\max}}{2}\|\x^{t+1}-\x^\star\|^2+2\mu_{\theta}\|\x^{t+1}-\x^\star\|^2 + \tfrac{2}{\mu_{\theta}}\|\blambda^{t+1}-\blambda^{t}\|^2,
    \end{aligned}
\end{equation}
where the first term on the RHS follows from $\z^t = \frac{1}{2}E_u \x^t$, and the second and third term are derived from \eqref{eqn:matrix_theta}. 
By adding \eqref{eqn:upper_bound_first2} and \eqref{eqn:upper_bound_second2} and using the definition of $\mathcal{H}$, we obtain an upper bound for $\eta \|\v^{t+1}-\v^{\star}\|^2_{\mathcal{H}}$ (the LHS of \eqref{eqn:thm_general_inequa}).
The lower bound for $\|\v^t-\v^{\star}\|^2_{\mathcal{H}}-\|\v^{t+1}-\v^{\star}\|^2_{\mathcal{H}}$ (the RHS of \eqref{eqn:thm_general_inequa}) can be derived from \eqref{eqn:bound_various_terms} as:
    \begin{align}
        &\|\v^t-\v^{\star}\|^2_{\mathcal{H}}-\|\v^{t+1}-\v^{\star}\|^2_{\mathcal{H}} \ge \notag \\
        &\tfrac{2m_fM_f}{m_f+M_f}\|\x^{t+1}-\x^\star\|^2+\tfrac{2}{m_f+M_f}\|\nabla f(\x^{t+1})-\nabla f(\x^\star)\|^2 \notag\\
        & -2(\x^{t+1}-\x^\star)^\top \e^t + \|\v^{t+1}-\v^t\|^2_{\mathcal{H}}. \label{eqn:lower_bound_following_term}
    \end{align}
By applying $ -2(\x^{t+1}-\x^\star)^\top \e^t \ge -\zeta\|\e^t\|^2 - \frac{1}{\zeta}\|\x^{t+1}-\x^\star\|^2$ (for any $\zeta > 0$) to \eqref{eqn:lower_bound_following_term}, and putting the upper bound (sum of \eqref{eqn:upper_bound_first2} and \eqref{eqn:upper_bound_second2}) and the lower bound \eqref{eqn:lower_bound_following_term} of \eqref{eqn:thm_general_inequa} together, it suffices to show that the following holds to establish \eqref{eqn:thm_general_inequa}:
\begin{align}
    &\eta \left\{\tfrac{5}{\mu_{\theta} \sigma^{+}_{\text{min}}}\Big( \E\left[\|\nabla f(\x^{t+1}) - \nabla f (\x^{\star})\|^2\right] + \E\left[ \|\x^{t+1} - \x^t\|^2_{\Gamma^2}\right]  \right. \notag \\ 
    & +\mu_{\theta}^2\E\left[\|\btheta^{t+1}-\btheta^t\|^2\right]+ \sigma^{L_u}_{\text{max}}\mu_{z}^2\E\left[\|\z^{t+1}-\z^t\|^2\right]\Big)\notag \\
    &+\tfrac{2}{\mu_{\theta}}\E\left[\|\blambda^{t+1} - \blambda^t\|^2\right] + \left(\tfrac{5}{\mu_{\theta} \sigma^{+}_{\text{min}}} +\tfrac{\zeta}{\eta}\right) \E\left[\|\e^t\|^2\right]\notag\\
    &\left.+ \left(2\mu_{\theta}+ \tfrac{\mu_{z} \sigma^{L_u}_{\text{max}}}{2}\right)\E\left[\|\x^{t+1} - \x^{\star}\|^2\right] + \E\left[\|\x^{t+1} - \x^{\star}\|^2_{\Gamma}\right]\right\} \notag \\
    &\le  \tfrac{2}{m_f+M_f} \E\left[\|\nabla f(\x^{t+1}) - \nabla f (\x^{\star})\|^2\right] + \E\left[\|\x^{t+1} - \x^{t}\|^2_{\Gamma}\right] \notag\\
    &+ \mu_{\theta} \E\left[\|\btheta^{t+1}-\btheta^t\|^2\right] + 2\mu_{z} \E\left[\|\z^{t+1}-\z^t\|^2\right] \notag\\
    & + \tfrac{1}{\mu_{\theta}}\E\left[\|\blambda^{t+1} - \blambda^t\|^2\right] + \tfrac{2}{\mu_{z}}\E\left[\|\balpha^{t+1} - \balpha^t\|^2\right] \notag\\
   & + \left(\tfrac{2m_f M_f}{m_f+M_f} - \tfrac{1}{\zeta}\right)\E\left[\|\x^{t+1} - \x^{\star}\|^2\right], \label{eqn:theorem_equation}
\end{align}
where we have applied $\E$ on both sides to account for the stochasticity of mini-batch selection.
An upper bound of $\E[\|\e^t\|^2]$ is obtained from \eqref{eqn:error_bound}.
We conclude by matching corresponding norm terms on both sides, we can see that the selection of $\eta$ in \eqref{eqn:eta_selection} guarantees that \eqref{eqn:theorem_equation} is satisfied. 
The proof of the asynchronous case follows the same line of analysis as in \cite[\textit{Theorem} 3]{li2022DRUID}, and is omitted due to length limitations.
\end{proof}
\begin{proof}[Proof of \textbf{\textit{Corollary}} \ref{lem:tune_epsilon}]
Observe that the dependency on $E_i$ in \eqref{eqn:eta_selection} is solely captured in the first term inside the \textit{min}, namely
\begin{equation}\label{eqn:eta_selection_first_term}
    \min_{i \in [n]}\frac{\mu_{\theta} \sigma^{+}_{\min}(\epsilon_i - \zeta\tau_i(E_i))}{5(\tau_i(E_i)+\epsilon_i^2)}.
\end{equation}
Since the rate is decreasing with $\eta$, we choose $\epsilon_i$ to maximize each term in \eqref{eqn:eta_selection_first_term} (this can be done for each agent locally), which admits an explicit solution as in \eqref{eqn:epsilon_selection}. 
Substituting in \eqref{eqn:eta_selection} gives \eqref{eqn:new_eta_selection}.
\end{proof}
\end{document}